%% file: SlabPerc_20160321.tex
\documentclass[11pt,a4paper]{article}
\usepackage[cp1251]{inputenc}
\usepackage{amsmath,amsthm,titlefoot}
\usepackage{amsfonts,amstext,amssymb,verbatim,epsfig}
\usepackage{dsfont}
\usepackage{enumerate}
\usepackage{psfrag}
\usepackage{graphicx}
\usepackage{color}
\usepackage{float}
\usepackage{subfig}

\usepackage[top=3cm, bottom=3.5cm, left=2.5cm, right=2.5cm]{geometry}

\def\N{{\mathbb{N}}}
\def\Z{{\mathbb{Z}}}
\def\P{{\mathbb{P}}}

\def\glue{{c_*}}
\def\matching{{C_*}}
\def\slab{{\mathbb S}}

\newtheorem{theorem}{Theorem}[section]
\newtheorem{corollary}[theorem]{Corollary}
\newtheorem{lemma}[theorem]{Lemma}
\newtheorem{proposition}[theorem]{Proposition}

\newtheoremstyle{likedef}
  {}%
  {}%
  {}%
  {}
  {\bfseries}%
  {.}%
  {.5em}%
  {}%

\theoremstyle{likedef}

\newtheorem{remark}[theorem]{Remark}

\numberwithin{equation}{section}

\begin{document}

\title{Crossing probabilities for critical Bernoulli percolation on slabs}

\author{
Deepan Basu
\thanks{Max-Planck Institute for Mathematics in the Sciences, 
Inselstrasse 22, 04103 Leipzig, Germany. 
email: deepan.basu@mis.mpg.de}
\and
Artem Sapozhnikov
\thanks{University of Leipzig, Department of Mathematics, 
Augustusplatz 10, 04109 Leipzig, Germany.
email: artem.sapozhnikov@math.uni-leipzig.de}
}

\maketitle\unmarkedfntext{\textbf{Bibliographic note:} After this project was completed, we learned that an alternative proof of our main result was independently obtained 
by Newman, Tassion, and Wu, which was subsequently written up in \cite{NTW15}. In the same paper, they prove that 
in the critical Bernoulli percolation on slabs, the probabilities of open left-right crossings of rectangles with any given aspect ratio 
are uniformly smaller than $1$, see Remark~\ref{rem:rsw}.}

\begin{abstract}
We prove that in the critical Bernoulli percolation on two dimensional lattice slabs 
the probabilities of open left-right crossings of rectangles with any given aspect ratio 
are uniformly positive. 
\end{abstract}

\section{Introduction}

One of the main tools in the study of planar percolation models at criticality is the Russo-Seymour-Welsh (RSW) theorem. 
It states that the probability that an open path connects the left and right sides of a rectangle is 
bounded away from $0$ and $1$ by constants that only depend on the aspect ratio of the rectangle. 
This theorem was first proved for critical Bernoulli percolation on planar lattices in \cite{R78,SW78,R81,K82} 
and recently has been extended to some other planar models, perhaps most notably to the FK-percolation \cite{DCHN11,DCST15} and 
Voronoi percolation \cite{BR,T14}. 

In this note we consider the critical Bernoulli percolation on two dimensional slabs $\Z^2\times\{0,\ldots,k-1\}^{d-2}$. 
We prove that the probability of crossing a rectangle is bounded from below by a positive constant which only depends on the aspect ratio of the rectangle 
and the slab parameter $k$, but does not depend on the size of the rectangle. 
Our work is inspired by a recent paper of Duminil-Copin, Sidoravicius, and Tassion \cite{DCST} 
where they prove the absence of percolation at criticality for slabs and 
develop techniques for ``glueing'' open paths. 
Our proof is partly based on these new ideas. 

\section{Notation and result}

Fix an integer $k\geq 1$, and define the slab of width $k$ by  
\[
\slab = \Z^2\times\{0,\ldots,k-1\}^{d-2}.
\]
We consider Bernoulli bond percolation on $\slab$ with parameter $p\in[0,1]$, and denote the corresponding measure by $\P_p$. 
Let $p_c$ be the critical threshold for percolation, i.e., 
\[
p_c = \inf\left\{p:\P_p[\text{open connected component of $0$ in $\slab$ is infinite}]>0\right\},
\]
and define the measure $\P = \P_{p_c}$. 

\medskip

For a subset $A$ of vertices of $\Z^2$, let 
\[
\overline{A} = A\times\{0,\ldots,k-1\}^{d-2}.
\]
Define a rectangle and its left and right boundary regions by
\[
B(m,n) = \overline{[0,m)\times[0,n)}, \qquad
L(m,n) = \overline{\{0\}\times[0,n)}, \qquad
R(m,n) = \overline{\{m-1\}\times[0,n)}. 
\]
Consider the crossing event 
\[
\mathrm{LR}(m,n) = \left\{\text{$L(m,n)$ is connected to $R(m,n)$ by an open path in $B(m,n)$}\right\}
\]
and the crossing probability 
\[
p(m,n) = \P\left[\mathrm{LR}(m,n)\right].
\]
In this note we prove the following theorem. 
\begin{theorem}\label{thm:rsw}
For any $\rho\in(0,\infty)$, 
\begin{equation}\label{eq:rsw}
\liminf_{n\to\infty}p(\lfloor \rho n\rfloor,n) >0.
\end{equation}
\end{theorem}
\begin{proof}
As some of the ``glueing'' ideas used in the proof are unnecessary for $k=1$ and easier for $k\geq 2$, 
we {\it assume from now on without further mentioning that $k\geq2$}.
The theorem is proved in $3$ steps:
\begin{itemize}\itemsep0pt
\item
The result holds for all $\rho\in(0,1)$. This is well known. We give a proof in Proposition~\ref{pr:rho<1}.
\item
If the result holds for some $\rho>1$, then it holds for all $\rho>1$. 
This is a well known fact in planar percolation. We prove the slab version in Proposition~\ref{pr:rho>1} using the planar approach together with 
a novel technique for glueing paths from \cite{DCST} (see Lemma~\ref{l:glueing}).
\item
There exist $c>0$ and $C<\infty$ such that for all $n\geq 1$, $p(44n,43n)\geq c\cdot p(43n,44n)^C$. 
This inequality is the main contribution of this paper. We prove it in Proposition~\ref{pr:shorttolong}.
\end{itemize}
\end{proof}
\begin{remark}\label{rem:rsw}
For $\rho<1$, the result of Theorem~\ref{thm:rsw} holds in any dimension $d\geq 2$. 
We believe that it also holds for $\rho\geq1$, but do not know a proof. 
Our method unfortunately crucially relies on planarity of slabs. 
If dimension is sufficently high, it is proved in \cite{A97} that the crossing probabilities tend to $1$ as $n\to\infty$. 
We believe that for percolation on slabs (and in low dimensions) for every $\rho>0$, $\limsup_{n\to\infty}p(\lfloor \rho n\rfloor,n) <1$, 
but do not have a proof yet.\footnotemark[1]\footnotetext[1]{For slabs, this was recently proved in \cite{NTW15}.}
\end{remark}

\medskip

Earlier we defined $\overline{A}$ as a subset of $\slab$ for each $A\subset\Z^2$. 
In the proofs we will often use the same notation $\overline{A}$ for $A\subset\slab$ meaning 
\[
\overline A = \{z=(z_1,\ldots,z_d)\in\slab~:~(z_1,z_2,x_3,\ldots,x_d)\in A\text{ for some $x_3,\ldots,x_d$ }\}.
\]
This way, for each $A\subset\Z^2$, $\overline{A}$ defined earlier is the same as $\overline{A\times\{0\}^{d-2}}$ defined just above.

\section{Crossings of narrow rectangles}

The following proposition is an adaptation to slabs of a well known fact about the probabilities of crossing hypercubes of fixed aspect ratio in the easy direction. 
\begin{proposition}\label{pr:rho<1}
For any $\rho\in(0,1)$, \eqref{eq:rsw} holds.
\end{proposition}
\begin{proof}
It suffices to prove that for all $L\in\N$, $\liminf_{n\to\infty}p(n,nL)>0$. 
The classical recursive inequality (applied to slabs) states that for every $p\in[0,1]$, $L,n\geq 1$, 
\[
\P_p\left[\mathrm{LR}(2n,2nL)\right]\leq \left((3L-1)\cdot \P_p\left[\mathrm{LR}(n,nL)\right]\right)^2,
\]
which implies that for all $s\geq 0$, 
\[
\P_p\left[\mathrm{LR}(2^s n,2^s nL)\right]\leq \left((3L-1)^2\cdot \P_p\left[\mathrm{LR}(n,nL)\right]\right)^{2^s}. 
\]
Fix $L\in\N$. If $\liminf_{n\to\infty}p(n,nL)=0$, there exists $n\in\N$ such that $p(n,nL)<\frac{1}{2(3L-1)^2}$. 
Since the crossing probability $\P_p\left[\mathrm{LR}(n,nL)\right]$ is continuous in $p$, 
there also exists $p>p_c$ such that $\P_p\left[\mathrm{LR}(n,nL)\right]< \frac{1}{2(3L-1)^2}$. 
For this choice of parameters, $\lim_{s\to\infty}\P_p\left[\mathrm{LR}(2^s n,2^s nL)\right] = 0$, which is impossible, 
since for every $p>p_c$, this limit equals to $1$ (see, e.g., \cite{G99}). Thus, $\liminf_{n\to\infty}p(n,nL)>0$.
\end{proof}

\section{Glueing}

In this section we recall a new technique for glueing paths from \cite{DCST}. 
It will be used to adapt some arguments from planar percolation to slabs. 
We begin with a classical combinatorial lemma about local modifications, see, e.g., \cite[Lemma~7]{DCST}. 
\begin{lemma}\label{l:AB}
Let $n\geq 1$ and $p\in(0,1)$. Let $A,B\subseteq \{0,1\}^n$ and $\mathbf P_p$ a product measure on $\{0,1\}^n$ with parameter $p$, i.e., 
\[
\mathbf P_p[\omega] = \prod_{i=1}^np^{\omega_i}(1-p)^{1-\omega_i},\quad \omega\in\{0,1\}^n.
\]
If there exists a map $f:A\to B$ such that for every $\omega'\in B$, there exists a set $S\subseteq\{1,\ldots, n\}$ such that $|S|\leq s$ and 
\[
\omega_i = \omega_i',\quad \text{for all }i\notin S\text{ and }\omega\in f^{-1}(\omega'),
\]
then 
\[
\mathbf P_p[A] \leq \left(\frac{2}{\min(p,1-p)}\right)^s
\cdot \mathbf P_p[B].
\]
\end{lemma}

\medskip

We will often apply Lemma~\ref{l:AB} in the case $p=p_c$ and $s$ being not bigger than the number of edges in $\overline{[-3,3]^2}$. 
Therefore, we define 
\[
\matching = \left(\frac{2}{\min(p,1-p)}\right)^{d\cdot 7^2\cdot k^{d-2}},\qquad
\glue = \frac{1}{1+\matching}.
\]

The following lemma is essentially proven in \cite[Lemma~6]{DCST}. 
\begin{lemma}\label{l:glueing}
Let $X_1$, $X_2$, $Y_1$, and $Y_2$ be disjoint connected subsets of the interior vertex boundary of $[0,m)\times[0,n)$ 
arranged in a counter-clockwise order. Then 
\[
\P\left[\text{$\overline{X_1}$ is connected to $\overline{X_2}$ in $B(m,n)$}\right]
\geq \glue\cdot 
\P\left[\begin{array}{c}\text{$\overline{X_1}$ is connected to $\overline{Y_1}$ in $B(m,n),$}\\
\text{$\overline{X_2}$ is connected to $\overline{Y_2}$ in $B(m,n)$}\end{array}\right]
\]
\end{lemma}
\begin{proof}
Let  
\begin{align*}
E_i &= \{\text{$\overline{X_i}$ is connected to $\overline{Y_i}$ in $B(m,n)$}\},\\ 
X &= \{\text{$\overline{X_1}$ is connected to $\overline{X_2}$ in $B(m,n)$}\},
\end{align*}
It suffices to prove that $\P[E_1\cap E_2\cap X^c] \leq \matching\cdot \P[X]$. 
We will use Lemma~\ref{l:AB}, where a suitable function $f:E_1\cap E_2\cap X^c\to X$ 
will be constructed using ideas from \cite[Lemma~6]{DCST}. 

\medskip

We fix an order on edges $\{e:|e| = 1\}$ in $\Z^d$ and enumerate all the vertices of $\slab$ arbitrarily. 
Define an order $<$ on self-avoiding paths from $\overline{X_1}$ to $\overline{Y_1}$ in $B(m,n)$ as follows. If $\gamma = (\gamma_0,\ldots, \gamma_n)$ and 
$\gamma' = (\gamma_0',\ldots,\gamma_{n'}')$ are two such paths, then $\gamma < \gamma'$ if
\begin{itemize}\itemsep0pt
\item
$\gamma_0$ has a smaller number than $\gamma_0'$, or 
\item
$n<n'$ and $\gamma = (\gamma_0',\ldots, \gamma_n')$, or
\item
there exists $k<\min(n,n')$ such that $(\gamma_0,\ldots,\gamma_k) = (\gamma_0',\ldots,\gamma_k')$, and 
the edge $\{0,\gamma_{k+1} - \gamma_k\}$ is smaller than $\{0,\gamma_{k+1}' - \gamma_k'\}$.
\end{itemize}
For $\omega\in E_1$, let $\gamma_{\min}(\omega)$ be the minimal open self-avoiding path from $\overline{X_1}$ to $\overline{Y_1}$ 
for the above defined order. Exactly the same construction as in the proof of \cite[Lemma~6, Fact~2]{DCST} gives a function 
$f:E_1\cap E_2\cap X^c\to E_1\cap X$ such that for every $\omega\in E_1\cap E_2\cap X^c$, there exists a $z\in \gamma_{\min}(\omega)$ such that 
\begin{itemize}\itemsep0pt
\item
$z\in\gamma_{\min}(f(\omega))$, 
\item
$z$ is a unique vertex on $\gamma_{\min}(f(\omega))$ connected to $\overline{X_2}$ by an open path that does not use edges of $\gamma_{\min}(f(\omega))$,
\item
$\omega_e = f(\omega)_e$ for all $e\notin \overline{z+ [-3,3]^2\times\{0\}^{d-2}}$.
\end{itemize}
Roughly speaking, one chooses $z\in\gamma_{\min}(\omega)$ such that $\overline{z}$ is connected to $\overline{X_2}$ and 
modifies the configuration in $\overline{z+ [-3,3]^2\times\{0\}^{d-2}}$ so that in the new configuration $z$ is connected 
to $\overline{X_1}$, $\overline{Y_1}$, and $\overline{X_2}$ by open paths having only vertex $z$ in common and so that the minimal open path from 
$\overline{X_1}$ to $\overline{Y_1}$ still passes through $z$. 
This tricky construction is very carefully explained in the proof of \cite[Lemma~6, Fact~2]{DCST}, therefore we do not repeat it here. 

The function $f$ satisfies the conditions of Lemma~\ref{l:AB} with $s$ being the number of edges in $\overline{[-3,3]^2}$. 
Thus, $\P[E_1\cap E_2\cap X^c] \leq \matching \cdot \P[E_1\cap X]\leq \matching \cdot \P[X]$. 
The proof of the lemma is complete. 
\end{proof}

\begin{corollary}\label{cor:glueing}
By Lemma~\ref{l:glueing} and the FKG inequality, 
\begin{multline*}
\P\left[\text{$\overline{X_1}$ is connected to $\overline{X_2}$ in $B(m,n)$}\right]\\
\geq \glue\cdot 
\P\left[\text{$\overline{X_1}$ is connected to $\overline{Y_1}$ in $B(m,n)$}\right]\cdot
\P\left[\text{$\overline{X_2}$ is connected to $\overline{Y_2}$ in $B(m,n)$}\right].
\end{multline*}
\end{corollary}

\section{Crossings of wide rectangles}
\begin{proposition}\label{pr:rho>1}
If \eqref{eq:rsw} holds for some $\rho>1$, then it holds for all $\rho>1$. 
\end{proposition}
\begin{proof}
This is immediate from the following inequality, which relates 
the crossing probability of a long rectangle with that of a shorter one. 
For all $m>n$, 
\begin{equation}\label{eq:rsw:rho>1}
p(2m-n,n) \geq \frac14\cdot \glue^3\cdot p(m,n)^4.
\end{equation}
The inequality \eqref{eq:rsw:rho>1} follows from two applications of Corollary~\ref{cor:glueing} illustrated on Figure~\ref{fig:glue1}. 
\begin{figure}[H]
\centering
\resizebox{14cm}{!}{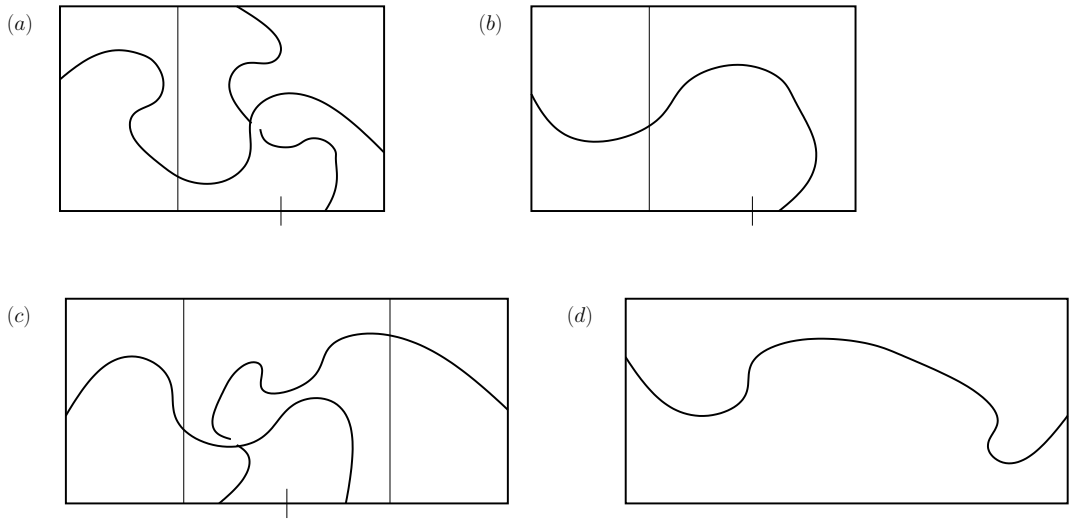}
\caption{(a) left-right crossing of $B(m,n)$ and top-bottom crossing of a $\overline{[m-n,m)\times [0,n)}$ landing on the right half of the bottom, 
(b) path from $L(m,n)$ to $\overline{[m-\frac{n}{2},m)\times\{0\}}$ in $B(m,n)$, 
(c) paths from $L(2m-n,n)$ to $\overline{[m-\frac{n}{2},m)\times\{0\}}$ and 
from $\overline{[m-n,m-\frac{n}{2})\times\{0\}}$ to $R(2m-n,n)$ in $B(2m-n,n)$, 
(d) left-right crossing of $B(2m-n,n)$.}
\label{fig:glue1}
\end{figure}
\end{proof}

\section{Crossings of rectangles: short and long directions}
The main contribution of this paper is the following proposition, which relates the crossing probability of 
a rectangle in the long direction with the one in the short. 
\begin{proposition}\label{pr:shorttolong}
For all $n\in\N$,  
\begin{equation}\label{eq:shorttolong}
p(44n,43n)\geq \frac{\glue^{21}\cdot p(43n,44n)^{198}}{10^{154}}.
\end{equation}
\end{proposition}
\begin{proof}
Fix $n\in\N$. 
We write 
\[
B = B(43n,44n), \quad
L = L(43n,44n), \quad
R = R(43n,44n), 
\]
and define 
\[
c = p(43n,44n), \qquad c' = \frac{\glue^{21}\cdot c^{198}}{10^{154}}.
\]
We prove the proposition by considering several cases. 
The first 2 steps are inspired by the ideas of Bollob\'as and Riordan from \cite{BR}, 
and aimed at restricting possible shapes of left-right crossings. 
Steps 3 and 4 contain preliminary estimates needed to implement the main idea in Step 5. 

\paragraph{Step 1.}
We first consider the case when there is a considerable probability that a left-right crossing of $B$ stays away from the top or bottom boundary of $B$, see Figure~\ref{fig:ass1}. 
Assume that 
\[
p(43n,42n) \geq \frac{c}{100}.
\]
\begin{figure}[H]
\centering
\resizebox{5cm}{!}{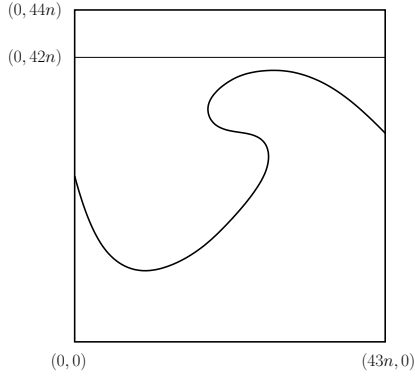}
\caption{Left-right crossing staying at least $2n$ away from the top of $B(43n,44n)$.}
\label{fig:ass1}
\end{figure}

Then by \eqref{eq:rsw:rho>1}, 
\[
p(44n,43n)\geq p(44n,42n) \geq \frac14\cdot \glue^3\cdot p(43n,42n)^4 \geq c', 
\]
which implies \eqref{eq:shorttolong}. 
Thus, we may assume that 
\begin{equation}\label{eq:assumption1}
p(43n,42n) < \frac{c}{100}.
\end{equation}

\paragraph{Step 2.}
Next, we consider the case when there is a considerable probability that a left-right crossing of $B$ starts sufficiently far away from the middle of $L$. Let
\begin{equation}\label{def:S}
S = \overline{\{0\}\times[20n,24n)}
\end{equation}
be the middle of $L$.
Assume that 
\[
\P\left[\text{$L\setminus S$ is connected to $R$ in $B$}\right]\geq \frac{c}{10}.
\]
Then, by reflectional symmetry, 
\[
\P\left[\text{$\overline{\{0\}\times[24n,44n)}$ is connected to $R$ in $B$}\right]\geq \frac{c}{20}.
\]
By assumption \eqref{eq:assumption1}, 
\[
\P\left[\text{$\overline{\{0\}\times[24n,44n)}$ is connected to $\overline{[0,43n)\times\{2n\}}$ in $B$}\right]\geq \frac{c}{20} - \frac{c}{100}\geq \frac{c}{100}.
\]
By rotational symmetry, the above display states precisely that 
\[
\P\left[\text{$\overline{\{0\}\times[0,43n)}$ is connected to $\overline{[22n,42n)\times\{0\}}$ in $B(42n,43n)$}\right]\geq \frac{c}{100}.
\]
Similarly to the second application of Corollary~\ref{cor:glueing} in the proof of \eqref{eq:rsw:rho>1}, see Figure~\ref{fig:ass2}, one gets 
\[
p(44n,43n) \geq \glue\cdot \left(\frac{c}{100}\right)^2 \geq c', 
\]
which is precisely \eqref{eq:shorttolong}. 
\begin{figure}[H]
\centering
\resizebox{12cm}{!}{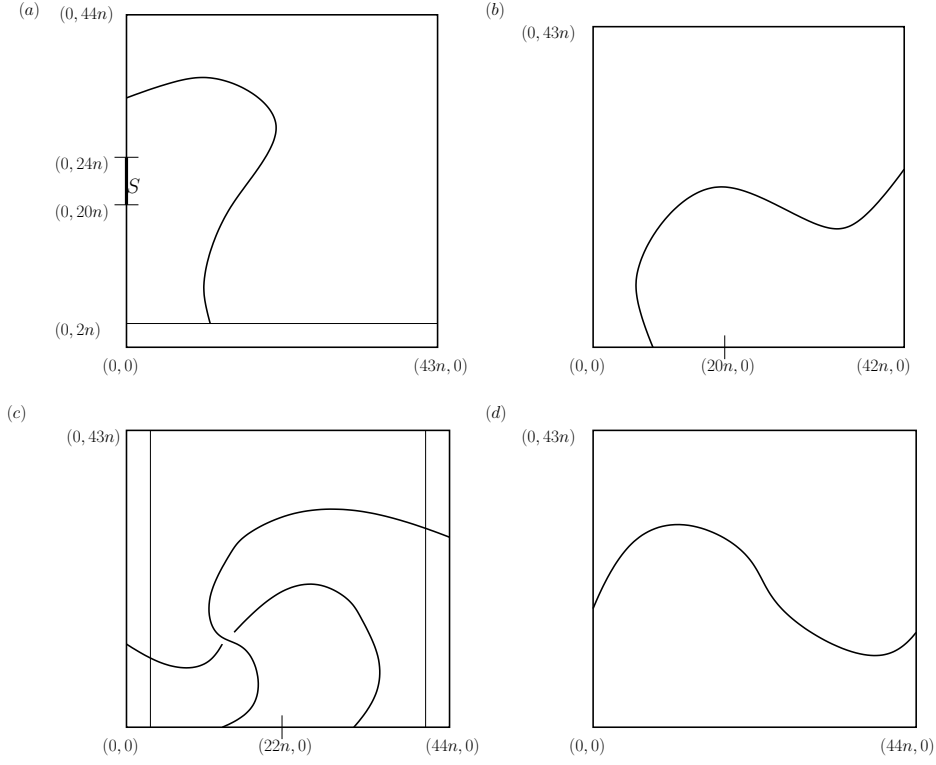}
\caption{(a) part of $L$ above $S$ is connected to $\overline{[0,43n)\times\{2n\}}$ in $B$, 
(b) rotation of (a) by $\frac{\pi}{2}$, 
(c) $L(44n,43n)$ is connected to $\overline{[22n,42n)\times\{0\}}$ and $\overline{[2n,22n)\times\{0\}}$ is connected to $R(44n,43n)$, 
(d) left-right crossing of $B(44n,43n)$.}
\label{fig:ass2}
\end{figure}
Thus, we may assume, in addition to \eqref{eq:assumption1}, that 
\begin{equation}\label{eq:assumption2}
\P\left[\text{$L\setminus S$ is connected to $R$ in $B$}\right]< \frac{c}{10}.
\end{equation}

\paragraph{Step 3.}
Here we consider the case when there is a considerable probability that two well-separated subsegments of $L$ are connected. 
For integers $a<b$, let
\[
T_{ab} = \overline{[0,43n)\times[a,b)}\quad\text{and}\quad
T = \overline{[0,43n)\times\Z}.
\]
Assume that for some $a<b$, 
\[
\P\left[\text{$\overline{\{0\}\times[0,4n)}$ is connected to $\overline{\{0\}\times[8n,12n)}$ in $T_{ab}$} \right]\geq \frac{\glue\cdot c^{18}}{10^{14}}.
\]
Then, by repetitive use of Corollary~\ref{cor:glueing}, see Figure~\ref{fig:ass3}, for each $m\geq 1$, 
\[
\P\left[\text{$\overline{\{0\}\times[0,4n)}$ is connected to $\overline{\{0\}\times[4n(m+1),4n(m+2))}$ in $T$} \right]\geq \frac{\glue^{2m-1}\cdot c^{18m}}{10^{14m}}.
\]
\begin{figure}[H]
\centering
\resizebox{12cm}{!}{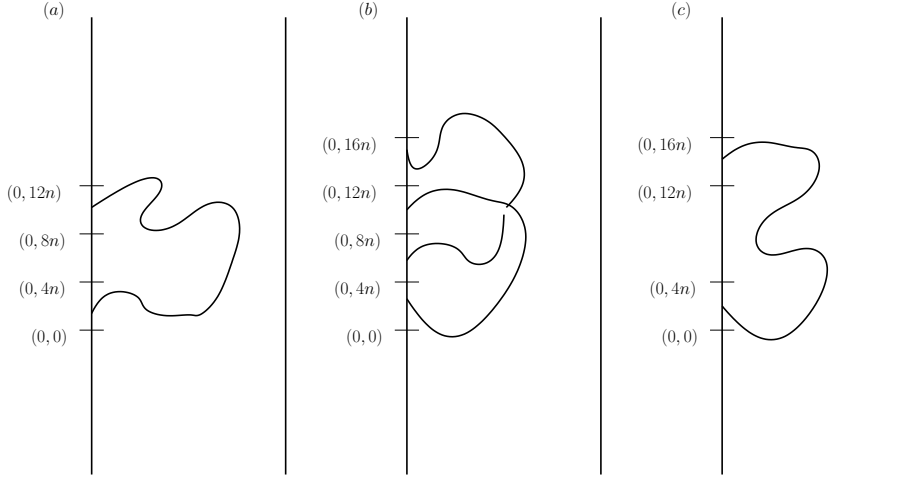}
\caption{Vertical extension of open paths.}
\label{fig:ass3}
\end{figure}
Note that if $m=11$, then the event on the left hand side implies that there is a vertical crossing of $\overline{[0,43n)\times[4n,48n)}$. Thus, 
\[
p(44n,43n) \geq \frac{\glue^{21}\cdot c^{198}}{10^{154}} = c',
\]
which gives \eqref{eq:shorttolong}. 
Therefore, we may assume, in addition to \eqref{eq:assumption1} and \eqref{eq:assumption2}, that 
\begin{equation}\label{eq:assumption3}
\P\left[\text{$\overline{\{0\}\times[0,4n)}$ is connected to $\overline{\{0\}\times[8n,12n)}$ in $T_{ab}$} \right]< \frac{\glue\cdot c^{18}}{10^{14}},\qquad\text{for all $a<b$}.
\end{equation}
Next, we derive several corollaries of assumption \eqref{eq:assumption3}.
\begin{corollary}\label{cor:assumption3:1}
Under the assumption \eqref{eq:assumption3}, for all $a<b$, 
\begin{equation}\label{eq:assumption3:1}
\P\left[\text{$\overline{\{0\}\times[8n,12n)}$ is connected to $\overline{\{43n-1\}\times[0,4n)}$ in $T_{ab}$} \right]< \frac{c^{9}}{10^{7}}.
\end{equation}
\end{corollary}
\begin{proof}[Proof of Corollary~\ref{cor:assumption3:1}]
Using reflectional symmetry and Corollary~\ref{cor:glueing}, 
\begin{multline*}
\P\left[\text{$\overline{\{0\}\times[8n,12n)}$ is connected to $\overline{\{43n-1\}\times[0,4n)}$ in $T_{ab}$} \right]^2\\
=
\P\left[\text{$\overline{\{0\}\times[8n,12n)}$ is connected to $\overline{\{43n-1\}\times[0,4n)}$ in $T_{ab}$} \right]\\
\qquad\qquad\qquad\cdot
\P\left[\text{$\overline{\{0\}\times[0,4n)}$ is connected to $\overline{\{43n-1\}\times[8n,12n)}$ in $T_{ab}$} \right]\\
\leq 
\glue^{-1}\cdot 
\P\left[\text{$\overline{\{0\}\times[0,4n)}$ is connected to $\overline{\{0\}\times[8n,12n)}$ in $T_{ab}$} \right] 
\stackrel{\eqref{eq:assumption3}}< \frac{c^{18}}{10^{14}}.
\end{multline*}
\begin{figure}[H]
\centering
\resizebox{12cm}{!}{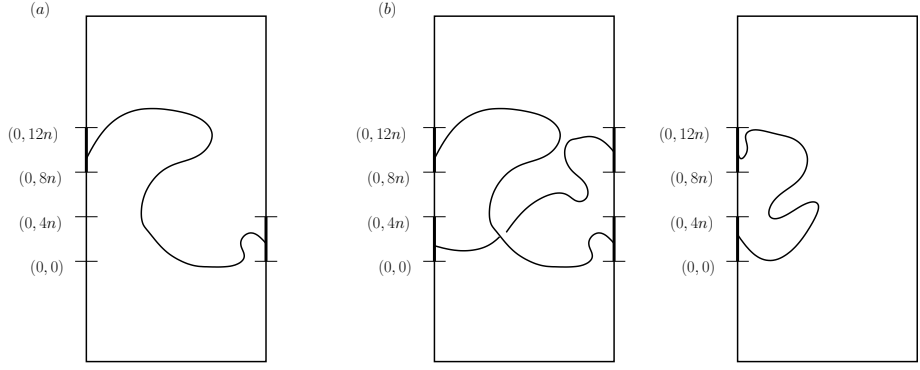}
\caption{(a) illustration of the event in \eqref{eq:assumption3:1}, (b) proof of Corollary~\ref{cor:assumption3:1}.}
\end{figure}
\end{proof}

\begin{corollary}\label{cor:assumption3:2}
Under the assumption \eqref{eq:assumption3}, for all $a<b$, 
\begin{equation}\label{eq:assumption3:2}
\P\left[
\begin{array}{c}
\text{there exist a simple path $\gamma$ from $\overline{\{0\}\times[0,4n)}$ to $\overline{\{43n-1\}\times[0,4n)}$ in $T_{ab}$}\\
\text{and a path $\gamma'$ from $\overline{\{0\}\times[8n,12n)}$ in $T_{ab}$, such that}\\
\text{the distance between $\overline\gamma$ and $\overline{\gamma'}$ is $\leq 2$}
\end{array}
\right]< \frac{3\cdot c^9}{10^7}.
\end{equation}
In particular,
\begin{equation}\label{eq:assumption3:2:T}
\P\left[
\begin{array}{c}
\text{there exist a simple path $\gamma$ from $\overline{\{0\}\times[0,4n)}$ to $\overline{\{43n-1\}\times[0,4n)}$ in $T$}\\
\text{and a path $\gamma'$ from $\overline{\{0\}\times[8n,12n)}$ in $T$, such that}\\
\text{the distance between $\overline\gamma$ and $\overline{\gamma'}$ is $\leq 2$}
\end{array}
\right]\leq \frac{3\cdot c^9}{10^7}.
\end{equation}
\end{corollary}
\begin{figure}[h]
\centering
\resizebox{4cm}{!}{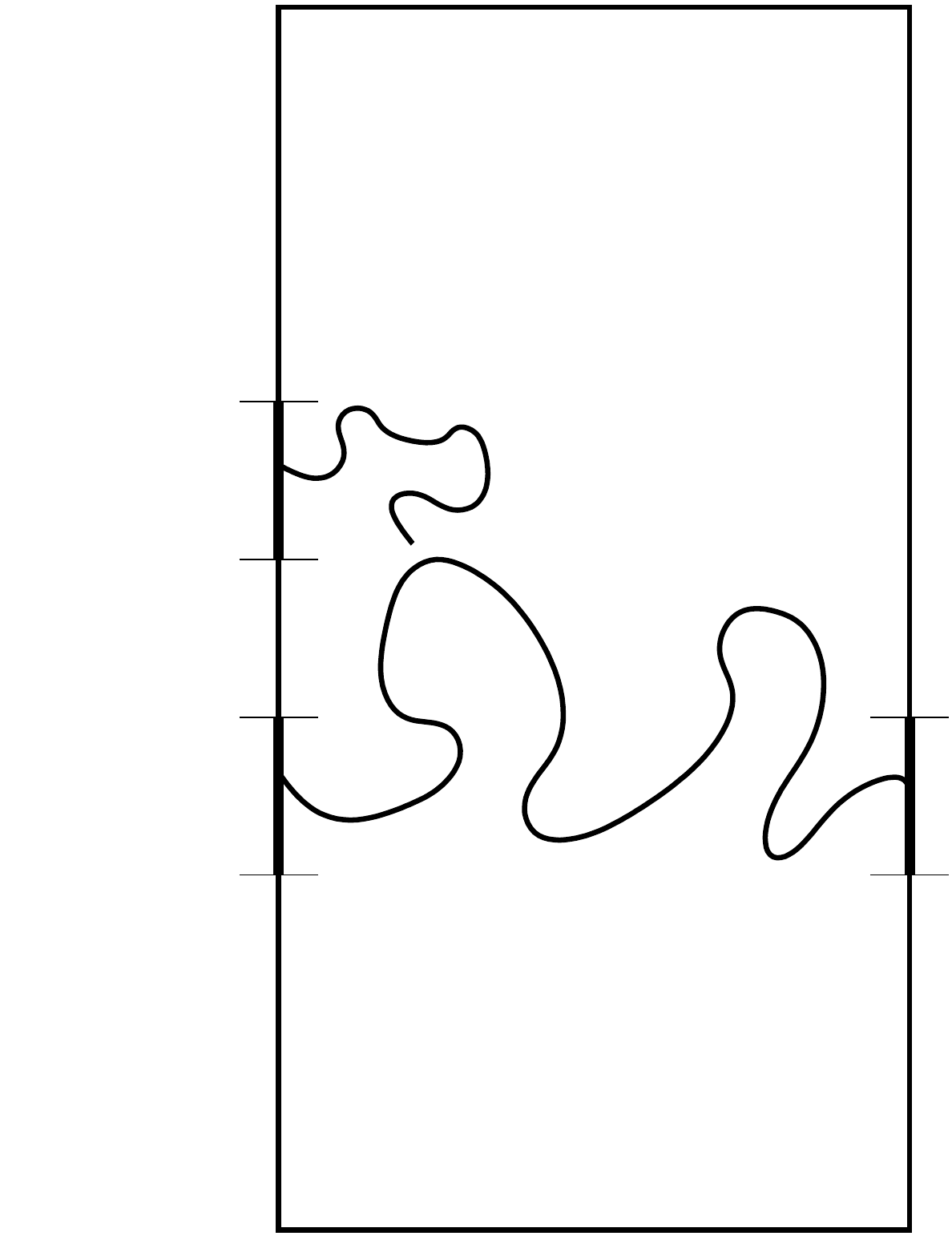}
\caption{An illustration of the event in \eqref{eq:assumption3:2}.}
\end{figure}
\begin{proof}[Proof of Corollary~\ref{cor:assumption3:2}]
It suffices to prove \eqref{eq:assumption3:2}, as \eqref{eq:assumption3:2:T} follows from \eqref{eq:assumption3:2} by sending $a\to-\infty$ and $b\to+\infty$. 

Denote the event in \eqref{eq:assumption3:2} by $A$. 
By the total probability formula, 
\begin{multline*}
\P[A]\leq 
\P\left[\text{$\overline{\{0\}\times[8n,12n)}$ is connected to $\overline{\{0\}\times[0,4n)}$ in $T_{ab}$} \right] \\
+ 
\P\left[\text{$\overline{\{0\}\times[8n,12n)}$ is connected to $\overline{\{43n-1\}\times[0,4n)}$ in $T_{ab}$} \right]\\
+ \P\left[
\begin{array}{c}
A, \quad \text{$\overline{\{0\}\times[8n,12n)}$ is not connected to $\overline{\{0\}\times[0,4n)}$ in $T_{ab}$},\\
\text{$\overline{\{0\}\times[8n,12n)}$ is not connected to $\overline{\{43n-1\}\times[0,4n)}$ in $T_{ab}$}
\end{array}
\right].
\end{multline*}
The sum of the first two probabilities is $<\frac{\glue\cdot c^{18}}{10^{14}} + \frac{c^{9}}{10^{7}}$, by the assumption \eqref{eq:assumption3} and \eqref{eq:assumption3:1}. 

Denote by $A'$ the event in the third probability. 
For a configuration $\omega$, let $P(\omega)$ be the set of vertices, which belong to at least one self-avoiding path from 
$\overline{\{0\}\times[0,4n)}$ to $\overline{\{43n-1\}\times[0,4n)}$ in $T_{ab}$, one may call it a backbone. 
Consider a local modification map $f$ from $A'$ to the event 
\[
A'' = \left\{ \omega''~:~ \begin{array}{c}\text{there exists a unique $z(\omega'')\in P(\omega'')$ connected to $\overline{\{0\}\times[8n,12n)}$}\\ 
\text{by an open path contained in $T_{ab}\setminus P(\omega'')$ except for the vertex $z(\omega'')$}\end{array}\right\}
\]
such that for all $\omega'\in A'$ and all $e\notin \overline{z(f(\omega')) + [-3,3]^2\times\{0\}^{d-2}}$, $f(\omega')_e = \omega_e'$. 
By Lemma~\ref{l:AB}, $\P[A'] \leq \matching\cdot \P[A'']\leq \glue^{-1}\cdot \P[A'']$. Since 
\[
A'' \subseteq \left\{\text{$\overline{\{0\}\times[0,4n)}$ is connected to $\overline{\{0\}\times[8n,12n)}$ in $T_{ab}$} \right\},
\]
we conclude that
\[
\P[A']\leq \glue^{-1}\cdot \P\left[\text{$\overline{\{0\}\times[0,4n)}$ is connected to $\overline{\{0\}\times[8n,12n)}$ in $T_{ab}$} \right]< \frac{c^{18}}{10^{14}},
\]
where the last inequality follows from the assumption \eqref{eq:assumption3}. 
Putting the bounds together, 
\[
\P[A]< \frac{\glue\cdot c^{18}}{10^{14}} + \frac{c^{9}}{10^{7}} + \frac{c^{18}}{10^{14}} \leq \frac{3\cdot c^9}{10^7}. 
\]
\end{proof}

\begin{corollary}\label{cor:assumption3:3}
Under the assumptions \eqref{eq:assumption2} and \eqref{eq:assumption3}, 
\begin{equation}\label{eq:assumption3:3}
\P\left[
\begin{array}{c}
\text{there exist a path $\gamma'$ from $\overline{\{0\}\times[0,4n)}$ in $T$}\\
\text{and a path $\gamma''$ from $\overline{\{0\}\times[16n,20n)}$ in $T$, such that}\\
\text{the distance between $\overline{\gamma'}$ and $\overline{\gamma''}$ is $\leq 4$}
\end{array}
\right]\leq \frac{12\cdot c^8}{10^7}.
\end{equation}
\end{corollary}
\begin{figure}[H]
\centering
\resizebox{4cm}{!}{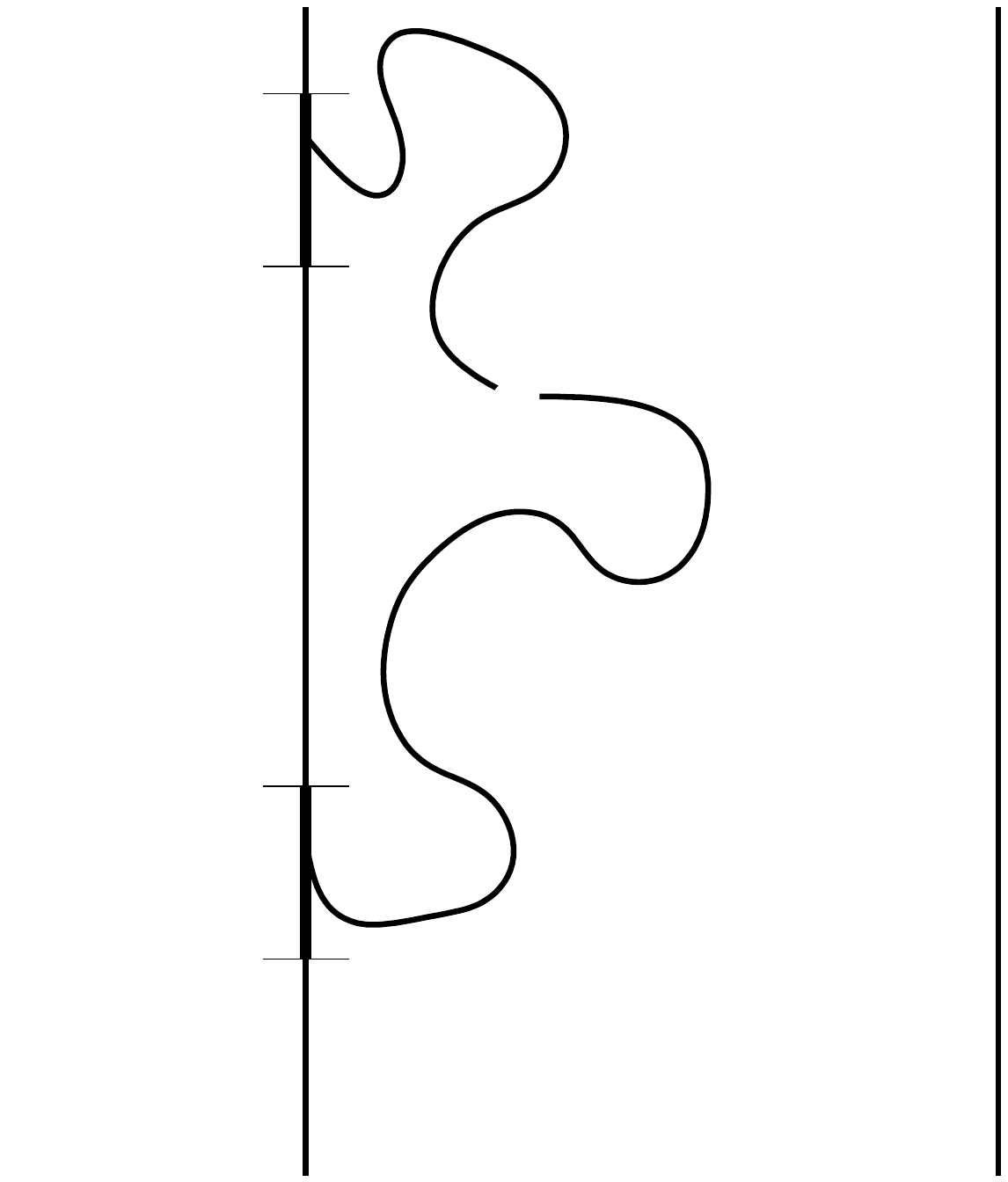}
\caption{An illustration of the event in \eqref{eq:assumption3:3}.}
\end{figure}
\begin{proof}[Proof of Corollary~\ref{cor:assumption3:3}]
Denote the event in \eqref{eq:assumption3:3} by $A$. 
By assumption \eqref{eq:assumption2}, 
\[
\P\left[\text{$\overline{\{0\}\times[8n,12n)}$ is connected to $\overline{\{43n-1\}\times[8n,12n)}$ in $T$}\right]\geq c - 2\, \frac{c}{10} \geq \frac{c}{2}.
\]
Since the above event and the event $A$ are increasing, by the FKG inequality, 
\[
\P[A] 
\leq \frac{2}{c}\cdot 
\P\left[A,\quad 
\text{$\overline{\{0\}\times[8n,12n)}$ is connected to $\overline{\{43n-1\}\times[8n,12n)}$ in $T$}
\right].
\]
The intersection of the two events on the right hand side implies that for any path $\gamma$ from $\overline{\{0\}\times[8n,12n)}$ to $\overline{\{43n-1\}\times[8n,12n)}$ in $T$, 
the distance from $\overline\gamma$ to $\overline{\gamma'}\cup\overline{\gamma''}$ is $\leq 2$. Thus, by \eqref{eq:assumption3:2:T}, 
\[
\P[A]\leq \frac{2}{c}\cdot 2\,\frac{3\cdot c^9}{10^7} = \frac{12\cdot c^8}{10^7}.
\]
\end{proof}

\paragraph{Step 4.}
The aim of this step is to introduce a certain event of positive probability, see Proposition~\ref{pr:mainevent}. 
Our choice of this event will be clarified in Step~5. 

Recall the definition of $S$ from \eqref{def:S}. 
For a configuration $\omega$, let $C_S = C_S(\omega)$ be the set of all $z\in T$ connected to $S$ by an open path in $T$. 
Let
\[
f(\omega) = 
\P\left[
\text{$\overline{\{0\}\times[4n,8n)}$ is connected to $\overline{\{43n-1\}\times\Z}$ in $T\setminus \overline{C_S(\omega)}$}
~\Big|~C_S(\omega)\right],
\]
and
\[
g(\omega) = 
\P\left[
\begin{array}{c}
\text{there exists a path $\gamma'$ from $\overline{\{0\}\times[4n,8n)}$ in $T$, such that}\\
\text{the distance between $\overline{\gamma'}$ and $\overline{C_S(\omega)}$ is $\leq 4$}
\end{array}
~\Big|~C_S(\omega)
\right].
\]
We consider the following events:
\begin{align*}
A_1 &= \left\{\text{$S$ is connected to $\overline{[0,43n)\times\{2n\}}$ in $T$}\right\},\\
A_2 &= \left\{f(\omega)\geq \frac{c^2}{10}\right\},\\
A_3 &= \left\{g(\omega)\leq \frac{c^4}{1000}\right\}.
\end{align*}
\begin{proposition}\label{pr:mainevent}
Under the assumptions \eqref{eq:assumption1}, \eqref{eq:assumption2}, and \eqref{eq:assumption3}, 
\[
\P[A_1\cap A_2\cap A_3]\geq \frac{c^4}{10^3}.
\]
\end{proposition}
\begin{proof}[Proof of Proposition~\ref{pr:mainevent}]
By assumptions \eqref{eq:assumption1} and \eqref{eq:assumption2}, 
\[
\P[A_1] \geq c - \frac{c}{10} - \frac{c}{100} \geq \frac{c}{2}. 
\]
By the Markov inequality and \eqref{eq:assumption3:3}, 
\[
\P[A_3^c] \leq \frac{1000}{c^4}\cdot \mathbb E[g] < \frac{1000}{c^4}\cdot\frac{12\cdot c^8}{10^7} = \frac{12\cdot c^4}{10^4}.
\]
To bound $\P[A_1\cap A_2]$ from below we use the Paley-Zygmund inequality. Using \eqref{eq:assumption3:3} and the FKG inequality, we first estimate
\begin{eqnarray*}
\mathbb E[f(\omega)\cdot \mathds{1}_{A_1}] 
&= 
&\P\left[
\begin{array}{c}
\text{$S$ is connected to $\overline{[0,43n)\times\{2n\}}$ in $T$, and }\\
\text{$\overline{\{0\}\times[4n,8n)}$ is connected to $\overline{\{43n-1\}\times\Z}$ in $T\setminus \overline{C_S(\omega)}$}
\end{array}
\right]\\[7pt]
&\geq 
&\P\left[
\begin{array}{c}
\text{$S$ is connected to $\overline{[0,43n)\times\{2n\}}$ in $T$, and }\\
\text{$\overline{\{0\}\times[4n,8n)}$ is connected to $\overline{\{43n-1\}\times\Z}$ in $T$}
\end{array}
\right] - \frac{12\cdot c^8}{10^7}\\[7pt]
&\geq 
&\P[A_1]\cdot \P\left[\text{$\overline{\{0\}\times[4n,8n)}$ is connected to $\overline{\{43n-1\}\times\Z}$ in $T$}\right] - \frac{12\cdot c^8}{10^7}\\[7pt]
&\geq 
&\frac{c}{2}\left(c - \frac{c}{10}\right) - \frac{12\cdot c^8}{10^7} \geq \frac{c^2}{5}.
\end{eqnarray*}
The Paley-Zygmund inequality for non-negative random variable $X$ states that $\mathbf P[X\geq \frac12 \mathbf E[X]] \geq \frac14\frac{(\mathbf E[X])^2}{\mathbf E[X^2]}$. 
We apply it to the measure $\mathbf P[\cdot] = \mathbb E \left[\mathds{1}_\cdot \frac{\mathds{1}_{A_1}}{\P[A_1]}\right]$, to get
\[
\mathbb E\left[\mathds{1}_{f(\omega)\geq \frac12\cdot\mathbb E[f(\omega)\cdot \frac{\mathds{1}_{A_1}}{\P[A_1]}]} \cdot \frac{\mathds{1}_{A_1}}{\P[A_1]}\right]
\geq \frac14\cdot \left(\mathbb E[f(\omega)\cdot \frac{\mathds{1}_{A_1}}{\P[A_1]}]\right)^2 .
\]
Thus, 
\[
\P[A_1\cap A_2] \geq \frac{c^4}{100},
\]
and we conclude that 
\[
\P[A_1\cap A_2\cap A_3] \geq \frac{c^4}{100} - \frac{12\cdot c^4}{10^4} \geq \frac{c^4}{10^3}.
\]
\end{proof}

\paragraph{Step 5.} We are ready to conclude. 
For a configuration $\omega$, let $Q(\omega)$ be the set of vertices from $T$, which are connected to $S$ by an open path in $\overline{[0,43n)\times[2n,\infty)}$.
Let $\Gamma(\omega)$ be the outer vertex boundary of $\overline{Q(\omega)}$, 
and $\Gamma'(\omega)$ the mirror reflection of $\Gamma$ with respect to the hyperplane $\{x~:~x_2 = 2n-\frac12\}$.
We denote the connected component of $T\setminus(\Gamma\cup\Gamma')$ which contains $0$ by $V$. 
Note that $V$ is finite for any $\omega\in A_1$. 

Let $X = \overline{\{0\}\times[4n,8n)}$, and $X' = \overline{\{0\}\times[-4n,0)}$. Note that $X'$ is the mirror reflection of $X$ with respect to the hyperplane $\{x~:~x_2 = 2n-\frac12\}$. 
Moreover, if $\omega\in A_2\cap A_3$, then both $X$ and $X'$ are contained in $V$. 
\begin{figure}[H]
\centering
\resizebox{5cm}{!}{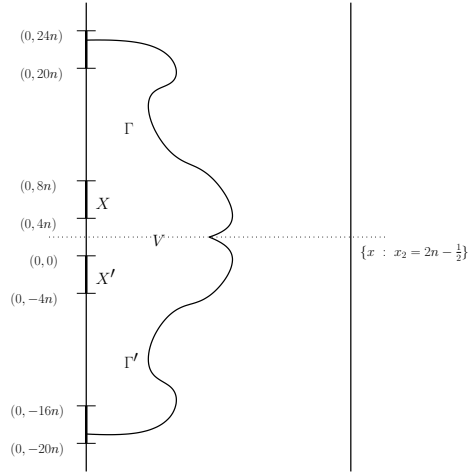}
\caption{An illustration of $\Gamma$, $\Gamma'$, and $V$ for a configuration from the event $A_1\cap A_2\cap A_3$. 
$\Gamma$ is the outer vertex boundary of the cluster of $S$ in $\overline{[0,43n)\times[2n,+\infty)}$, 
$\Gamma'$ is its mirror reflection with respect to the hyperplane $\{x~:~x_2 = 2n-\frac12\}$, 
and $V$ is the connected component of $T\setminus(\Gamma\cup\Gamma')$ containing the origin.}
\end{figure}

We consider an auxiliary probability space $\Omega'$ with configurations $\omega'$ and the same probability measure $\P$ on it. 
We compute 
\begin{multline*}
\P\left[\text{$X$ is connected to $X'$ in $T$ by an open path in $\omega'$}\right]\\[7pt]
\geq 
\P\otimes\P\left[(\omega,\omega')~:~
\begin{array}{c}
\omega\in A_1\cap A_2\cap A_3,\\
\text{$X$ is connected to $X'$ in $V(\omega)$ by an open path in $\omega'$}
\end{array}
\right]\\[7pt]
\stackrel{(*)}\geq 
\matching^{-1}\cdot 
\P\otimes\P\left[(\omega,\omega')~:~
\begin{array}{c}
\omega\in A_1\cap A_2\cap A_3,\\
\text{$X$ is {\it not} connected to $X'$ in $V(\omega)$ by an open path in $\omega'$}\\
\text{$X$ is connected to $\Gamma'(\omega)$ in $V(\omega)$ by an open path in $\omega'$,}\\
\text{$X'$ is connected to $\Gamma(\omega)$ in $V(\omega)$ by an open path in $\omega'$,}\\
\text{there is no open path $\pi$ in $\omega'$ from $X$ in $V(\omega)$}\\
\text{so that the distance between $\overline\pi$ and $\Gamma(\omega)$ is $\leq 4$,}\\
\text{there is no open path $\pi'$ in $\omega'$ from $X'$ in $V(\omega)$}\\
\text{so that the distance between $\overline{\pi'}$ and $\Gamma'(\omega)$ is $\leq 4$}
\end{array}
\right]\\[7pt]
\geq
\matching^{-1}\cdot 
\mathbb E_\omega\left[
\mathds{1}_{A_1\cap A_2\cap A_3}(\omega)\cdot 
\P_{\omega'}\left[
\begin{array}{c}
\text{$X$ is connected to $\Gamma'(\omega)$ in $V(\omega)$ by an open path in $\omega'$,}\\
\text{$X'$ is connected to $\Gamma(\omega)$ in $V(\omega)$ by an open path in $\omega'$}
\end{array}
\right]\right]\\[7pt]
- 
\matching^{-1}\cdot 
\mathbb E_\omega\left[
\mathds{1}_{A_1\cap A_2\cap A_3}(\omega)\cdot 
\P_{\omega'}\left[
\begin{array}{c}
\text{there is an open path $\pi$ in $\omega'$ from $X$ in $V(\omega)$}\\
\text{so that the distance between $\overline\pi$ and $\Gamma(\omega)$ is $\leq 4$,}\\
\text{or}\\
\text{there is an open path $\pi'$ in $\omega'$ from $X'$ in $V(\omega)$}\\
\text{so that the distance between $\overline{\pi'}$ and $\Gamma'(\omega)$ is $\leq 4$}
\end{array}
\right] 
\right]\\[7pt]
- \matching^{-1}\cdot \P\left[\text{$X$ is connected to $X'$ in $T$ by an open path in $\omega'$}\right]\\[7pt]
\geq
\matching^{-1}\cdot 
\mathbb E_\omega\left[
\mathds{1}_{A_1\cap A_2\cap A_3}(\omega)\cdot 
\left[f(\omega)^2 - 2g(\omega)\right]\right]\\[7pt]
- \matching^{-1}\cdot \P\left[\text{$X$ is connected to $X'$ in $T$ by an open path in $\omega'$}\right].
\end{multline*}
The inequality ($*$) follows from Lemma~\ref{l:AB} and a similar local transformation as in the proof of Corollary~\ref{cor:assumption3:2}. 
(Mind that every path from $X$ to $\Gamma'$ in $V$ and every path from $X'$ to $\Gamma$ have intersecting projections, and 
all the ``intersection points'' are sufficiently far away from $\Gamma\cup\Gamma'$ to allow for a local modification far away from $\Gamma\cup\Gamma'$.)
The last inequality comes from the FKG inequality and the definitions of event $A_1$ and functions $f$ and $g$. 

By the definition of events $A_2$ and $A_3$ and Proposition~\ref{pr:mainevent}, 
\[
\P\left[\text{$X$ is connected to $X'$ in $T$}\right]
\geq
\glue\cdot \left(\frac{c^4}{100} - 2\cdot \frac{c^4}{1000}\right)\cdot \frac{c^4}{10^3}\,. 
\]
In particular, there exist $a<b$ such that 
\[
\P\left[\text{$X$ is connected to $X'$ in $T_{ab}$}\right]
\geq \frac{\glue\cdot c^8}{10^6}.
\]
From this we conclude, as in the argument of Step~3, that $p(44n,43n)\geq c'$ (or simply observe that the above inequality contradicts the assumption \eqref{eq:assumption3}). 

\medskip

The proof of Proposition~\ref{pr:shorttolong} is complete. 
\end{proof}

\end{document}

%% file: glue1.pdf_t
\begin{picture}(0,0)%
\includegraphics{glue1.pdf}%
\end{picture}%
\setlength{\unitlength}{4144sp}%
\begingroup\makeatletter\ifx\SetFigFont\undefined%
\gdef\SetFigFont#1#2#3#4#5{%
  \reset@font\fontsize{#1}{#2pt}%
  \fontfamily{#3}\fontseries{#4}\fontshape{#5}%
  \selectfont}%
\fi\endgroup%
\begin{picture}(16248,7920)(-914,-9388)
\put(-899,-1861){\makebox(0,0)[lb]{\smash{{\SetFigFont{20}{24.0}{\rmdefault}{\mddefault}{\updefault}{\color[rgb]{0,0,0}$(a)$}%
}}}}
\put(-899,-6361){\makebox(0,0)[lb]{\smash{{\SetFigFont{20}{24.0}{\rmdefault}{\mddefault}{\updefault}{\color[rgb]{0,0,0}$(c)$}%
}}}}
\put(6301,-1861){\makebox(0,0)[lb]{\smash{{\SetFigFont{20}{24.0}{\rmdefault}{\mddefault}{\updefault}{\color[rgb]{0,0,0}$(b)$}%
}}}}
\put(7651,-6361){\makebox(0,0)[lb]{\smash{{\SetFigFont{20}{24.0}{\rmdefault}{\mddefault}{\updefault}{\color[rgb]{0,0,0}$(d)$}%
}}}}
\end{picture}%

%% file: ass1.pdf_t
\begin{picture}(0,0)%
\includegraphics{ass1.pdf}%
\end{picture}%
\setlength{\unitlength}{4144sp}%
\begingroup\makeatletter\ifx\SetFigFont\undefined%
\gdef\SetFigFont#1#2#3#4#5{%
  \reset@font\fontsize{#1}{#2pt}%
  \fontfamily{#3}\fontseries{#4}\fontshape{#5}%
  \selectfont}%
\fi\endgroup%
\begin{picture}(7158,7020)(-1274,-7816)
\put(-449,-7711){\makebox(0,0)[lb]{\smash{{\SetFigFont{20}{24.0}{\rmdefault}{\mddefault}{\updefault}{\color[rgb]{0,0,0}$(0,0)$}%
}}}}
\put(5401,-7711){\makebox(0,0)[lb]{\smash{{\SetFigFont{20}{24.0}{\rmdefault}{\mddefault}{\updefault}{\color[rgb]{0,0,0}$(43n,0)$}%
}}}}
\put(-1259,-1051){\makebox(0,0)[lb]{\smash{{\SetFigFont{20}{24.0}{\rmdefault}{\mddefault}{\updefault}{\color[rgb]{0,0,0}$(0,44n)$}%
}}}}
\put(-1259,-1951){\makebox(0,0)[lb]{\smash{{\SetFigFont{20}{24.0}{\rmdefault}{\mddefault}{\updefault}{\color[rgb]{0,0,0}$(0,42n)$}%
}}}}
\end{picture}%

%% file: ass2.pdf_t
\begin{picture}(0,0)%
\includegraphics{ass2.pdf}%
\end{picture}%
\setlength{\unitlength}{4144sp}%
\begingroup\makeatletter\ifx\SetFigFont\undefined%
\gdef\SetFigFont#1#2#3#4#5{%
  \reset@font\fontsize{#1}{#2pt}%
  \fontfamily{#3}\fontseries{#4}\fontshape{#5}%
  \selectfont}%
\fi\endgroup%
\begin{picture}(17148,14310)(-3164,-15016)
\put(-1349,-7711){\makebox(0,0)[lb]{\smash{{\SetFigFont{20}{24.0}{\rmdefault}{\mddefault}{\updefault}{\color[rgb]{0,0,0}$(0,0)$}%
}}}}
\put(4501,-7711){\makebox(0,0)[lb]{\smash{{\SetFigFont{20}{24.0}{\rmdefault}{\mddefault}{\updefault}{\color[rgb]{0,0,0}$(43n,0)$}%
}}}}
\put(-2159,-1051){\makebox(0,0)[lb]{\smash{{\SetFigFont{20}{24.0}{\rmdefault}{\mddefault}{\updefault}{\color[rgb]{0,0,0}$(0,44n)$}%
}}}}
\put(-2249,-3886){\makebox(0,0)[lb]{\smash{{\SetFigFont{20}{24.0}{\rmdefault}{\mddefault}{\updefault}{\color[rgb]{0,0,0}$(0,24n)$}%
}}}}
\put(-899,-4336){\makebox(0,0)[lb]{\smash{{\SetFigFont{34}{40.8}{\rmdefault}{\mddefault}{\updefault}{\color[rgb]{0,0,0}$S$}%
}}}}
\put(-2249,-4786){\makebox(0,0)[lb]{\smash{{\SetFigFont{20}{24.0}{\rmdefault}{\mddefault}{\updefault}{\color[rgb]{0,0,0}$(0,20n)$}%
}}}}
\put(-2249,-7036){\makebox(0,0)[lb]{\smash{{\SetFigFont{20}{24.0}{\rmdefault}{\mddefault}{\updefault}{\color[rgb]{0,0,0}$(0,2n)$}%
}}}}
\put(6526,-1411){\makebox(0,0)[lb]{\smash{{\SetFigFont{20}{24.0}{\rmdefault}{\mddefault}{\updefault}{\color[rgb]{0,0,0}$(0,43n)$}%
}}}}
\put(7426,-7711){\makebox(0,0)[lb]{\smash{{\SetFigFont{20}{24.0}{\rmdefault}{\mddefault}{\updefault}{\color[rgb]{0,0,0}$(0,0)$}%
}}}}
\put(9901,-7711){\makebox(0,0)[lb]{\smash{{\SetFigFont{20}{24.0}{\rmdefault}{\mddefault}{\updefault}{\color[rgb]{0,0,0}$(20n,0)$}%
}}}}
\put(12826,-7711){\makebox(0,0)[lb]{\smash{{\SetFigFont{20}{24.0}{\rmdefault}{\mddefault}{\updefault}{\color[rgb]{0,0,0}$(42n,0)$}%
}}}}
\put(-2924,-961){\makebox(0,0)[lb]{\smash{{\SetFigFont{20}{24.0}{\rmdefault}{\mddefault}{\updefault}{\color[rgb]{0,0,0}$(a)$}%
}}}}
\put(5851,-961){\makebox(0,0)[lb]{\smash{{\SetFigFont{20}{24.0}{\rmdefault}{\mddefault}{\updefault}{\color[rgb]{0,0,0}$(b)$}%
}}}}
\put(-1349,-14911){\makebox(0,0)[lb]{\smash{{\SetFigFont{20}{24.0}{\rmdefault}{\mddefault}{\updefault}{\color[rgb]{0,0,0}$(0,0)$}%
}}}}
\put(-2024,-9061){\makebox(0,0)[lb]{\smash{{\SetFigFont{20}{24.0}{\rmdefault}{\mddefault}{\updefault}{\color[rgb]{0,0,0}$(0,43n)$}%
}}}}
\put(1576,-14911){\makebox(0,0)[lb]{\smash{{\SetFigFont{20}{24.0}{\rmdefault}{\mddefault}{\updefault}{\color[rgb]{0,0,0}$(22n,0)$}%
}}}}
\put(4726,-14911){\makebox(0,0)[lb]{\smash{{\SetFigFont{20}{24.0}{\rmdefault}{\mddefault}{\updefault}{\color[rgb]{0,0,0}$(44n,0)$}%
}}}}
\put(7426,-14911){\makebox(0,0)[lb]{\smash{{\SetFigFont{20}{24.0}{\rmdefault}{\mddefault}{\updefault}{\color[rgb]{0,0,0}$(0,0)$}%
}}}}
\put(13276,-14911){\makebox(0,0)[lb]{\smash{{\SetFigFont{20}{24.0}{\rmdefault}{\mddefault}{\updefault}{\color[rgb]{0,0,0}$(44n,0)$}%
}}}}
\put(6526,-9061){\makebox(0,0)[lb]{\smash{{\SetFigFont{20}{24.0}{\rmdefault}{\mddefault}{\updefault}{\color[rgb]{0,0,0}$(0,43n)$}%
}}}}
\put(5851,-8611){\makebox(0,0)[lb]{\smash{{\SetFigFont{20}{24.0}{\rmdefault}{\mddefault}{\updefault}{\color[rgb]{0,0,0}$(d)$}%
}}}}
\put(-3149,-8611){\makebox(0,0)[lb]{\smash{{\SetFigFont{20}{24.0}{\rmdefault}{\mddefault}{\updefault}{\color[rgb]{0,0,0}$(c)$}%
}}}}
\end{picture}%

%% file: ass3.pdf_t
\begin{picture}(0,0)%
\includegraphics{ass3.pdf}%
\end{picture}%
\setlength{\unitlength}{4144sp}%
\begingroup\makeatletter\ifx\SetFigFont\undefined%
\gdef\SetFigFont#1#2#3#4#5{%
  \reset@font\fontsize{#1}{#2pt}%
  \fontfamily{#3}\fontseries{#4}\fontshape{#5}%
  \selectfont}%
\fi\endgroup%
\begin{picture}(16923,8883)(-1139,-7744)
\put(-674,-5236){\makebox(0,0)[lb]{\smash{{\SetFigFont{20}{24.0}{\rmdefault}{\mddefault}{\updefault}{\color[rgb]{0,0,0}$(0,0)$}%
}}}}
\put(-899,-4336){\makebox(0,0)[lb]{\smash{{\SetFigFont{20}{24.0}{\rmdefault}{\mddefault}{\updefault}{\color[rgb]{0,0,0}$(0,4n)$}%
}}}}
\put(-899,-3436){\makebox(0,0)[lb]{\smash{{\SetFigFont{20}{24.0}{\rmdefault}{\mddefault}{\updefault}{\color[rgb]{0,0,0}$(0,8n)$}%
}}}}
\put(-1124,-2536){\makebox(0,0)[lb]{\smash{{\SetFigFont{20}{24.0}{\rmdefault}{\mddefault}{\updefault}{\color[rgb]{0,0,0}$(0,12n)$}%
}}}}
\put(4726,-1636){\makebox(0,0)[lb]{\smash{{\SetFigFont{20}{24.0}{\rmdefault}{\mddefault}{\updefault}{\color[rgb]{0,0,0}$(0,16n)$}%
}}}}
\put(4726,-2536){\makebox(0,0)[lb]{\smash{{\SetFigFont{20}{24.0}{\rmdefault}{\mddefault}{\updefault}{\color[rgb]{0,0,0}$(0,12n)$}%
}}}}
\put(4951,-3436){\makebox(0,0)[lb]{\smash{{\SetFigFont{20}{24.0}{\rmdefault}{\mddefault}{\updefault}{\color[rgb]{0,0,0}$(0,8n)$}%
}}}}
\put(4951,-4336){\makebox(0,0)[lb]{\smash{{\SetFigFont{20}{24.0}{\rmdefault}{\mddefault}{\updefault}{\color[rgb]{0,0,0}$(0,4n)$}%
}}}}
\put(5176,-5236){\makebox(0,0)[lb]{\smash{{\SetFigFont{20}{24.0}{\rmdefault}{\mddefault}{\updefault}{\color[rgb]{0,0,0}$(0,0)$}%
}}}}
\put(11026,-5236){\makebox(0,0)[lb]{\smash{{\SetFigFont{20}{24.0}{\rmdefault}{\mddefault}{\updefault}{\color[rgb]{0,0,0}$(0,0)$}%
}}}}
\put(10801,-4336){\makebox(0,0)[lb]{\smash{{\SetFigFont{20}{24.0}{\rmdefault}{\mddefault}{\updefault}{\color[rgb]{0,0,0}$(0,4n)$}%
}}}}
\put(10576,-2536){\makebox(0,0)[lb]{\smash{{\SetFigFont{20}{24.0}{\rmdefault}{\mddefault}{\updefault}{\color[rgb]{0,0,0}$(0,12n)$}%
}}}}
\put(10576,-1636){\makebox(0,0)[lb]{\smash{{\SetFigFont{20}{24.0}{\rmdefault}{\mddefault}{\updefault}{\color[rgb]{0,0,0}$(0,16n)$}%
}}}}
\put(-449,884){\makebox(0,0)[lb]{\smash{{\SetFigFont{20}{24.0}{\rmdefault}{\mddefault}{\updefault}{\color[rgb]{0,0,0}$(a)$}%
}}}}
\put(5401,884){\makebox(0,0)[lb]{\smash{{\SetFigFont{20}{24.0}{\rmdefault}{\mddefault}{\updefault}{\color[rgb]{0,0,0}$(b)$}%
}}}}
\put(11206,884){\makebox(0,0)[lb]{\smash{{\SetFigFont{20}{24.0}{\rmdefault}{\mddefault}{\updefault}{\color[rgb]{0,0,0}$(c)$}%
}}}}
\end{picture}%

%% file: cor1.pdf_t
\begin{picture}(0,0)%
\includegraphics{cor1.pdf}%
\end{picture}%
\setlength{\unitlength}{4144sp}%
\begingroup\makeatletter\ifx\SetFigFont\undefined%
\gdef\SetFigFont#1#2#3#4#5{%
  \reset@font\fontsize{#1}{#2pt}%
  \fontfamily{#3}\fontseries{#4}\fontshape{#5}%
  \selectfont}%
\fi\endgroup%
\begin{picture}(18273,7308)(-1139,-7069)
\put(-674,-5236){\makebox(0,0)[lb]{\smash{{\SetFigFont{20}{24.0}{\rmdefault}{\mddefault}{\updefault}{\color[rgb]{0,0,0}$(0,0)$}%
}}}}
\put(-899,-4336){\makebox(0,0)[lb]{\smash{{\SetFigFont{20}{24.0}{\rmdefault}{\mddefault}{\updefault}{\color[rgb]{0,0,0}$(0,4n)$}%
}}}}
\put(-899,-3436){\makebox(0,0)[lb]{\smash{{\SetFigFont{20}{24.0}{\rmdefault}{\mddefault}{\updefault}{\color[rgb]{0,0,0}$(0,8n)$}%
}}}}
\put(-1124,-2536){\makebox(0,0)[lb]{\smash{{\SetFigFont{20}{24.0}{\rmdefault}{\mddefault}{\updefault}{\color[rgb]{0,0,0}$(0,12n)$}%
}}}}
\put(6301,-5236){\makebox(0,0)[lb]{\smash{{\SetFigFont{20}{24.0}{\rmdefault}{\mddefault}{\updefault}{\color[rgb]{0,0,0}$(0,0)$}%
}}}}
\put(6076,-4336){\makebox(0,0)[lb]{\smash{{\SetFigFont{20}{24.0}{\rmdefault}{\mddefault}{\updefault}{\color[rgb]{0,0,0}$(0,4n)$}%
}}}}
\put(6076,-3436){\makebox(0,0)[lb]{\smash{{\SetFigFont{20}{24.0}{\rmdefault}{\mddefault}{\updefault}{\color[rgb]{0,0,0}$(0,8n)$}%
}}}}
\put(5851,-2536){\makebox(0,0)[lb]{\smash{{\SetFigFont{20}{24.0}{\rmdefault}{\mddefault}{\updefault}{\color[rgb]{0,0,0}$(0,12n)$}%
}}}}
\put(12376,-5236){\makebox(0,0)[lb]{\smash{{\SetFigFont{20}{24.0}{\rmdefault}{\mddefault}{\updefault}{\color[rgb]{0,0,0}$(0,0)$}%
}}}}
\put(12151,-4336){\makebox(0,0)[lb]{\smash{{\SetFigFont{20}{24.0}{\rmdefault}{\mddefault}{\updefault}{\color[rgb]{0,0,0}$(0,4n)$}%
}}}}
\put(12151,-3436){\makebox(0,0)[lb]{\smash{{\SetFigFont{20}{24.0}{\rmdefault}{\mddefault}{\updefault}{\color[rgb]{0,0,0}$(0,8n)$}%
}}}}
\put(11926,-2536){\makebox(0,0)[lb]{\smash{{\SetFigFont{20}{24.0}{\rmdefault}{\mddefault}{\updefault}{\color[rgb]{0,0,0}$(0,12n)$}%
}}}}
\put(-674,-16){\makebox(0,0)[lb]{\smash{{\SetFigFont{20}{24.0}{\rmdefault}{\mddefault}{\updefault}{\color[rgb]{0,0,0}$(a)$}%
}}}}
\put(6301,-16){\makebox(0,0)[lb]{\smash{{\SetFigFont{20}{24.0}{\rmdefault}{\mddefault}{\updefault}{\color[rgb]{0,0,0}$(b)$}%
}}}}
\end{picture}%

%% file: cor2.pdf_t
\begin{picture}(0,0)%
\includegraphics{cor2.pdf}%
\end{picture}%
\setlength{\unitlength}{4144sp}%
\begingroup\makeatletter\ifx\SetFigFont\undefined%
\gdef\SetFigFont#1#2#3#4#5{%
  \reset@font\fontsize{#1}{#2pt}%
  \fontfamily{#3}\fontseries{#4}\fontshape{#5}%
  \selectfont}%
\fi\endgroup%
\begin{picture}(5427,7041)(-1139,-7069)
\put(-674,-5236){\makebox(0,0)[lb]{\smash{{\SetFigFont{20}{24.0}{\rmdefault}{\mddefault}{\updefault}{\color[rgb]{0,0,0}$(0,0)$}%
}}}}
\put(-899,-4336){\makebox(0,0)[lb]{\smash{{\SetFigFont{20}{24.0}{\rmdefault}{\mddefault}{\updefault}{\color[rgb]{0,0,0}$(0,4n)$}%
}}}}
\put(-899,-3436){\makebox(0,0)[lb]{\smash{{\SetFigFont{20}{24.0}{\rmdefault}{\mddefault}{\updefault}{\color[rgb]{0,0,0}$(0,8n)$}%
}}}}
\put(-1124,-2536){\makebox(0,0)[lb]{\smash{{\SetFigFont{20}{24.0}{\rmdefault}{\mddefault}{\updefault}{\color[rgb]{0,0,0}$(0,12n)$}%
}}}}
\put(2251,-5101){\makebox(0,0)[lb]{\smash{{\SetFigFont{20}{24.0}{\rmdefault}{\mddefault}{\updefault}{\color[rgb]{0,0,0}$\gamma$}%
}}}}
\put(1666,-2671){\makebox(0,0)[lb]{\smash{{\SetFigFont{20}{24.0}{\rmdefault}{\mddefault}{\updefault}{\color[rgb]{0,0,0}$\gamma'$}%
}}}}
\end{picture}%

%% file: cor3.pdf_t
\begin{picture}(0,0)%
\includegraphics{cor3.pdf}%
\end{picture}%
\setlength{\unitlength}{4144sp}%
\begingroup\makeatletter\ifx\SetFigFont\undefined%
\gdef\SetFigFont#1#2#3#4#5{%
  \reset@font\fontsize{#1}{#2pt}%
  \fontfamily{#3}\fontseries{#4}\fontshape{#5}%
  \selectfont}%
\fi\endgroup%
\begin{picture}(5223,6141)(-1139,-6169)
\put(-674,-5236){\makebox(0,0)[lb]{\smash{{\SetFigFont{20}{24.0}{\rmdefault}{\mddefault}{\updefault}{\color[rgb]{0,0,0}$(0,0)$}%
}}}}
\put(-899,-4336){\makebox(0,0)[lb]{\smash{{\SetFigFont{20}{24.0}{\rmdefault}{\mddefault}{\updefault}{\color[rgb]{0,0,0}$(0,4n)$}%
}}}}
\put(-1124,-736){\makebox(0,0)[lb]{\smash{{\SetFigFont{20}{24.0}{\rmdefault}{\mddefault}{\updefault}{\color[rgb]{0,0,0}$(0,20n)$}%
}}}}
\put(-1124,-1636){\makebox(0,0)[lb]{\smash{{\SetFigFont{20}{24.0}{\rmdefault}{\mddefault}{\updefault}{\color[rgb]{0,0,0}$(0,16n)$}%
}}}}
\put(1846,-961){\makebox(0,0)[lb]{\smash{{\SetFigFont{20}{24.0}{\rmdefault}{\mddefault}{\updefault}{\color[rgb]{0,0,0}$\gamma''$}%
}}}}
\put(1801,-3391){\makebox(0,0)[lb]{\smash{{\SetFigFont{20}{24.0}{\rmdefault}{\mddefault}{\updefault}{\color[rgb]{0,0,0}$\gamma'$}%
}}}}
\end{picture}%

%% file: V.pdf_t
\begin{picture}(0,0)%
\includegraphics{V.pdf}%
\end{picture}%
\setlength{\unitlength}{4144sp}%
\begingroup\makeatletter\ifx\SetFigFont\undefined%
\gdef\SetFigFont#1#2#3#4#5{%
  \reset@font\fontsize{#1}{#2pt}%
  \fontfamily{#3}\fontseries{#4}\fontshape{#5}%
  \selectfont}%
\fi\endgroup%
\begin{picture}(9048,11316)(-1589,-9319)
\put(-1574,-5236){\makebox(0,0)[lb]{\smash{{\SetFigFont{20}{24.0}{\rmdefault}{\mddefault}{\updefault}{\color[rgb]{0,0,0}$(0,-4n)$}%
}}}}
\put(-1574,-8836){\makebox(0,0)[lb]{\smash{{\SetFigFont{20}{24.0}{\rmdefault}{\mddefault}{\updefault}{\color[rgb]{0,0,0}$(0,-20n)$}%
}}}}
\put(-1574,-7936){\makebox(0,0)[lb]{\smash{{\SetFigFont{20}{24.0}{\rmdefault}{\mddefault}{\updefault}{\color[rgb]{0,0,0}$(0,-16n)$}%
}}}}
\put(-1124,-4336){\makebox(0,0)[lb]{\smash{{\SetFigFont{20}{24.0}{\rmdefault}{\mddefault}{\updefault}{\color[rgb]{0,0,0}$(0,0)$}%
}}}}
\put(-1349,-3436){\makebox(0,0)[lb]{\smash{{\SetFigFont{20}{24.0}{\rmdefault}{\mddefault}{\updefault}{\color[rgb]{0,0,0}$(0,4n)$}%
}}}}
\put(-1349,-2536){\makebox(0,0)[lb]{\smash{{\SetFigFont{20}{24.0}{\rmdefault}{\mddefault}{\updefault}{\color[rgb]{0,0,0}$(0,8n)$}%
}}}}
\put(-1349,164){\makebox(0,0)[lb]{\smash{{\SetFigFont{20}{24.0}{\rmdefault}{\mddefault}{\updefault}{\color[rgb]{0,0,0}$(0,20n)$}%
}}}}
\put(-1349,1064){\makebox(0,0)[lb]{\smash{{\SetFigFont{20}{24.0}{\rmdefault}{\mddefault}{\updefault}{\color[rgb]{0,0,0}$(0,24n)$}%
}}}}
\put(451,-2986){\makebox(0,0)[lb]{\smash{{\SetFigFont{34}{40.8}{\rmdefault}{\mddefault}{\updefault}{\color[rgb]{0,0,0}$X$}%
}}}}
\put(6751,-4111){\makebox(0,0)[lb]{\smash{{\SetFigFont{20}{24.0}{\rmdefault}{\mddefault}{\updefault}{\color[rgb]{0,0,0}$\{x~:~x_2=2n-\frac{1}{2}\}$}%
}}}}
\put(1126,-1186){\makebox(0,0)[lb]{\smash{{\SetFigFont{34}{40.8}{\rmdefault}{\mddefault}{\updefault}{\color[rgb]{0,0,0}$\Gamma$}%
}}}}
\put(1126,-6811){\makebox(0,0)[lb]{\smash{{\SetFigFont{34}{40.8}{\rmdefault}{\mddefault}{\updefault}{\color[rgb]{0,0,0}$\Gamma'$}%
}}}}
\put(451,-4786){\makebox(0,0)[lb]{\smash{{\SetFigFont{34}{40.8}{\rmdefault}{\mddefault}{\updefault}{\color[rgb]{0,0,0}$X'$}%
}}}}
\put(1801,-3886){\makebox(0,0)[lb]{\smash{{\SetFigFont{34}{40.8}{\rmdefault}{\mddefault}{\updefault}{\color[rgb]{0,0,0}$V$}%
}}}}
\end{picture}%